\documentclass[a4paper]{amsart}
\usepackage{amsfonts, latexsym, amssymb, graphicx,
multicol, mathrsfs, color, amscd, verbatim, paralist,
xspace, url, euscript, stmaryrd,  amsmath, amscd, enumitem,
bbold, multirow, tikz, mathtools, mathabx}
\usepackage[all,pdf,cmtip]{xy}

\usepackage[colorlinks, linkcolor=blue, citecolor=magenta, urlcolor=cyan]{hyperref}

\usepackage{mathabx}

\def\ra{\rightarrow}

\newtheorem{theorem}{THEOREM}[section]
\newtheorem{corollary}[theorem]{Corollary}
\newtheorem{proposition}[theorem]{Proposition}

\newtheorem{example}[theorem]{Example}
\theoremstyle{definition}

\theoremstyle{remark}
\newtheorem{remark}[theorem]{Remark}

\newcommand\CC{{\mathbb C}}
\newcommand\RR{{\mathbb R}}

\newcommand\ZZ{{\mathbb Z}}
\newcommand\NN{{\mathbb N}}

\def\GL{\mathop{\rm GL}\nolimits}

\def\SU{\mathop{\rm SU}\nolimits}
\def\U{\mathop{\rm U}\nolimits}
\def\Aut{\mathop{\rm Aut}\nolimits}

\makeatletter
\def\blfootnote{\xdef\@thefnmark{}\@footnotetext}
\makeatother

\begin{document}

\title[Actions of  the unitary group on quotients of Hopf manifolds]{Effective transitive actions
\vspace{0.1cm}\\
of  the unitary group
\vspace{0.1cm}\\
on quotients of Hopf manifolds}
\author[Isaev]{Alexander Isaev}

\address{Mathematical Sciences Institute\\
Australian National University\\
Acton, ACT 2601, Australia}
\email{alexander.isaev@anu.edu.au}

\maketitle

\thispagestyle{empty}

\pagestyle{myheadings}

\begin{abstract} In article \cite{IKru} we showed that every connected complex manifold of dimension $n\ge 2$ that admits an effective transitive action by holomorphic transformations of the unitary group $\U_n$ is biholomorphic to the quotient of a Hopf manifold by the action of $\ZZ_m$ for some integer $m$ satisfying $(n,m)=1$. In this note, we complement the above result with an explicit description of all effective transitive actions of $\U_n$ on such quotients, which provides an answer to a 10-year old question.
\end{abstract}

\section{Introduction}\label{intro}
\setcounter{equation}{0}

In \cite{IKru} we classified all connected complex manifolds of dimension $n\ge 2$ that admit an effective action of the unitary group $\U_n$ by holomorphic transformations. In particular, in Theorem 4.5 of \cite{IKru} we showed that if the action of $\U_n$ is transitive, then the manifold in question is biholomorphic to a certain quotient of a Hopf manifold. Recall that for a nonzero complex number $d$ with $|d|\ne 1$ the corresponding $n$-dimensional Hopf manifold, say $M_d^n$, is constructed by identifying every point $z\in\CC^n\setminus\{0\}$ with the point $d\cdot z$ (hence with each of the points $d^k\cdot z$ for $k\in\ZZ$). Clearly, $M_d^n$ is compact. Further, letting $[z]\in M_d^n$ be the equivalence class of $z\in\CC^n\setminus\{0\}$, for $m\in \NN$ we denote by $M_d^n/\ZZ_m$ the compact complex manifold obtained from $M_d^n$ by identifying $[z]$ with $[e^{\frac{2\pi i}{m}}z]$  (hence with each of the points $[e^{\frac{2\pi iK}{m}}z]$ for $K=0,\dots,m-1$). Theorem 4.5 of \cite{IKru} then asserts that if the group $\U_n$ acts on a connected complex manifold $X$ of dimension $n\ge 2$ effectively and transitively by holomorphic transformations, then $X$ is biholomorphic to $M_d^n/\ZZ_m$ for some $m\in\NN$ satisfying $(n,m)=1$. 

Conversely, as the following example shows (cf.~\cite[Example 4.1]{IKru}), every manifold $M_d^n/\ZZ_m$ with $(n,m)=1$ admits an effective transitive action of $\U_n$.  

\begin{example}\label{example}\rm First, we define an action of $\U_n$ on $M_d^n$. Let $\lambda$ be a complex number such that $e^{\frac{2\pi(\lambda-i)}{n}}=d$. Write any element $A\in \U_n$ as $A=e^{it}\cdot B$, where $t\in \RR$ and $B\in \SU_n$, and set
\begin{equation}
A\,[z]:=[e^{\lambda t}\cdot Bz].\label{action}
\end{equation}
To verify that this action is well-defined, consider any other representation of $A$ as above, i.e., $A=e^{i(t+\frac{2\pi k}{n}+2\pi \ell)}\cdot(e^{-\frac{2\pi ik}{n}}B)$ for some $\ell\in\ZZ$, $k\in\{0,\dots,n-1\}$. Then formula (\ref{action}) yields
$$
A[z]=[e^{\lambda(t+\frac{2\pi k}{n}+2\pi \ell)}\cdot e^{-\frac{2\pi ik}{n}}Bz]=
[d^{k+n\ell}e^{\lambda t}\cdot Bz]=[e^{\lambda t}\cdot Bz]
$$
as required.

Next, action (\ref{action}) is transitive and effective. Indeed, to verify its effectiveness, let $(e^{it}\cdot B)[z]=[z]$ for some $t\in\RR$, $B\in \SU_n$ and all
$z\in\CC^n\setminus\{0\}$. Then for some $k\in\{0,\dots,n-1\}$ we have $B=e^{\frac{2\pi ik}{n}}\cdot\hbox{id}$, and there exists $\ell\in\ZZ$ such that $e^{\lambda t}\cdot e^{\frac{2\pi i k}{n}}=d^{\ell}$. Therefore, using the definition of $\lambda$ we obtain
$$
\displaystyle t=\frac{2\pi \ell}{n},\quad e^{\frac{2\pi i k}{n}}=e^{-\frac{2\pi i \ell}{n}},
$$
which implies $e^{it}\cdot B=\hbox{id}$ as claimed.

Now, fix $m\in\NN$ with $(n,m)=1$, consider the quotient $M_d^n/\ZZ_m$ and for $[z]\in M_d^n$ denote by $\{[z]\}\in M_d^n/\ZZ_m$ the equivalence class of $[z]$. We define an action of $\U_n$ on $M_d^n/\ZZ_m$ by the formula  $A\{[z]\}:=\{A[z]\}$ for $A\in \U_n$. This action is clearly transitive. To verify its effectiveness, let $(e^{it}\cdot B)\{[z]\}=\{[z]\}$ for some $t\in\RR$, $B\in \SU_n$ and all $z\in\CC^n\setminus\{0\}$. Then for some $k\in\{0,\dots,n-1\}$ we have $B=e^{\frac{2\pi ik}{n}}\cdot\hbox{id}$, and there exist $\ell\in\ZZ$, $K\in\{0,\dots,m-1\}$  such that $e^{\lambda t}\cdot e^{\frac{2\pi i k}{n}}=e^{\frac{2\pi i K}{m}}d^{\ell}$. Using the definition of $\lambda$ we then see
\begin{equation}
\displaystyle t=\frac{2\pi \ell}{n},\quad e^{\frac{2\pi i k}{n}}=e^{-\frac{2\pi i \ell}{n}+\frac{2\pi i K}{m}}.\label{quotienteffectiveness}
\end{equation}
Since $(n,m)=1$, the second identity in (\ref{quotienteffectiveness}) yields $K=0$, and we obtain $e^{it}\cdot B=\hbox{id}$ as required.
\end{example}

Since the time article \cite{IKru} appeared in print, we have been asked---on numerous occasions---whether there exists a reasonable complete description of all effective transitive actions of $\U_n$ by holomorphic transformations on $M_d^n/\ZZ_m$ where one assumes $(n,m)=1$. In this note we, somewhat belatedly, give a positive answer to this question, which is now more than ten years old. 

Our results are summarized as follows. First, we obtain:

\begin{theorem}\label{hopf} If for some $m\in \NN$ the quotient $M_d^n/\ZZ_m$ admits an effective action of $\U_n$ by holomorphic transformations, then $(n,m)=1$. Further, every effective transitive action of $\U_n$ on $M_d^n/\ZZ_m$ by holomorphic transformations has either the form
\begin{equation}
A\{[z]\}=\left\{\left[e^{i\left(1+(p+\frac{q}{m})n\right)t}d^{\frac{nrt}{2\pi}}CBC^{-1}z\right]\right\}\label{type1}
\end{equation}
or the form
\begin{equation}
A\{[z]\}=\left\{\left[e^{i\left(-1+(p+\frac{q}{m})n\right)t}d^{\frac{nrt}{2\pi}}C\overline{B}C^{-1}z\right]\right\},\label{type2}
\end{equation}
where $A\in \U_n$ is represented as $A=e^{it}\cdot B$ with $t\in\RR$
and $B\in \SU_n$, $p,q,r\in\ZZ$, $r\ne 0$, $C\in \GL_n(\CC)$, and for
any $\mu\in\RR$ we set $d^{\mu}:=|d|^{\mu}e^{i\mu\hbox{\tiny arg}\,d}$. 
\end{theorem}

\begin{remark}\label{dimtwo}\rm For $n=2$ every action of the form
  (\ref{type2}) is in fact an action of the form (\ref{type1}) since complex conjugation $B\mapsto\overline{B}$ is an inner automorphism of $\SU_2$; it coincides with the conjugation by the element $\begin{pmatrix} 0 & i\\ i & 0 \end{pmatrix}$. 
\end{remark}

\begin{remark}\label{power} In the statement of Theorem \ref{hopf} we chose to define $d^{\mu}$ for $\mu\in\RR$ as $d^{\mu}:=|d|^{\mu}e^{i\mu\hbox{\tiny arg}\,d}$, where $\hbox{arg}$ is the ordinary argument function taking values in the semiopen interval $[0,2\pi)$. Any other possible definition of $d^{\mu}$ differs from ours by a factor of the form $e^{2\pi i\mu L}$, where $L\in\ZZ$ is independent of $\mu$. Notice that for any such alternative definition of $d^{\mu}$ formulas (\ref{type1}), (\ref{type2}) remain valid with $p$ replaced by $p-Lr$.       
\end{remark}

Observe that in Theorem \ref{hopf} we do not claim that either of formulas (\ref{type1}), (\ref{type2}) always defines an effective action despite the fact that these formulas were obtained under the effectiveness assumption. Necessary and sufficient conditions on the integers $p,q,r$ that guarantee effectiveness are given below:

\begin{proposition}\label{fin} We have:

\noindent {\rm (i)} action {\rm (\ref{type1})} is effective if and only if
there do not exist
$\ell\in\ZZ$, $K\in\{0,\dots,m-1\}$ such that the following conditions are satisfied:
$$
\begin{array}{ll}
{\rm (a)} & \displaystyle \frac{\ell}{r}\left(1+\left(p+\frac{q}{m}\right)n\right)-\frac{nK}{m}\in\ZZ,\\
\vspace{-0.1cm}\\
{\rm (b)} & \displaystyle \frac{\ell}{r}\left(p+\frac{q}{m}\right)-\frac{K}{m}\not\in\ZZ;
\end{array}
$$

\noindent {\rm (ii)} action {\rm (\ref{type2})} is effective if and only
if there do not exist
$\ell\in\ZZ$, $K\in\{0,\dots,m-1\}$ such that {\rm (b)} and the following condition are satisfied:
$$
\begin{array}{ll}
{\rm (c)} & \displaystyle \frac{\ell}{r}\left(-1+\left(p+\frac{q}{m}\right)n\right)-\frac{nK}{m}\in\ZZ.
\end{array}
$$
\end{proposition}

Finally, as one can see from the following corollary, the formulas from Theorem \ref{hopf} and the conditions from
Proposition \ref{fin} significantly simplify in the case of Hopf manifolds (i.e., when $m=1$):

\begin{corollary}\label{purehopf} Every effective transitive action of $\U_n$ on $M_d^n$ by holomorphic transformations has either the form
\begin{equation}
A[z]=\left[e^{i\left(1+pn\right)t}d^{\frac{nrt}{2\pi}}CBC^{-1}z\right]\label{type11}
\end{equation}
or the form 
\begin{equation}
A[z]=\left[e^{i\left(-1+pn\right)t}d^{\frac{nrt}{2\pi}}C\overline{B}C^{-1}z\right],\label{type22}
\end{equation}
where $A\in \U_n$ is represented as $A=e^{it}\cdot B$ with $t\in\RR$
and $B\in \SU_n$, $p,r\in\ZZ$, $r\ne 0$, $C\in \GL_n(\CC)$. Furthermore, action {\rm (\ref{type11})} is effective if and only if there does not
exist $\ell\in\ZZ$ such that $r$ divides $\ell(1+pn)$ but does not divide $\ell p$.
Action {\rm (\ref{type22})} is effective if and only if there does not
exist $\ell\in\ZZ$ such that $r$ divides $\ell(-1+pn)$ but does not divide
$\ell p$.
\end{corollary} 

\noindent It is an easy exercise to show that action (\ref{action}) from Example \ref{example} can be written in the form (\ref{type11}) with $C=\hbox{id}$, some $p\in\ZZ$ and $r=1$.

\section{The proofs}\label{results}
\setcounter{equation}{0}

\begin{proof}[Proof of Theorem {\rm \ref{hopf}}] Recall that for any connected compact complex manifold $X$ the group $\Aut(X)$ of its holomorphic automorphisms is a complex Lie group in the compact-open topology (see \cite{BM}). As noted in the proof of Theorem 4.5 in
\cite{IKru}, the complex Lie group $\Aut(M_d^n/\ZZ_m)$ is naturally isomorphic to the complex Lie group ${\mathcal G}:=(\GL_n(\CC)/H)/\ZZ_m$, where $H:=\{d^k\cdot\hbox{id},\,k\in\ZZ\}\subset\GL_n(\CC)$ and $\ZZ_m$ is identified with the subgroup of $\GL_n(\CC)/H$ that consists of all elements of the form $e^{\frac{2\pi i K}{m}}H$, $K\in\{0,\dots,m-1\}$. 

We will now find the maximal compact subgroups of ${\mathcal G}$. First, consider\linebreak $G:=\{d^t\cdot\hbox{id},t\in\RR\}\subset \GL_n(\CC)$. Then the subgroup ${\mathbf K}:=G/H\subset\GL_n(\CC)/H$ is isomorphic to $S^1$ (with the isomorphism given by $(d^t\cdot\hbox{id})H\mapsto e^{2\pi it}$). We view ${\mathbf K}$ as a subgroup of ${\mathcal G}$ by way of the embedding $g\mapsto g\ZZ_m$, $g\in {\mathbf K}$. Further, taking into account the natural embedding of $\U_n$ in $\GL_n(\CC)/H$ by the map $g\mapsto gH$, $g\in \U_n$, we conclude that $(\U_n/\ZZ_m)\cdot {\mathbf K}$ is a maximal compact subgroup of ${\mathcal G}$, where we notice that $(\U_n/\ZZ_m)\cap {\mathbf K}=\{e\}$. Since ${\mathbf K}$ lies in the center of ${\mathcal G}$, any other maximal compact subgroup of ${\mathcal G}$ has the form $s_0(\U_n/\ZZ_m)s_0^{-1}\cdot {\mathbf K}$ for some $s_0\in{\mathcal G}$.

Suppose now that we are given an effective action of $\U_n$ on
$M_d^n/\ZZ_m$ by holomorphic transformations. Clearly, the action
induces an embedding $\tau: \U_n\ra {\mathcal G}$. Since $\tau(\U_n)$ is a
compact subgroup of ${\mathcal G}$, we have $\tau(\U_n)\subset
s_0(\U_n/\ZZ_m)s_0^{-1}\cdot {\mathbf K}$ for some $s_0\in {\mathcal G}$. Consider
the restriction of $\tau$ to $\SU_n$. Since there does not exist a
nontrivial homomorphism of $\SU_n$ to $S^1$ (which follows, e.g., from \cite[Lemma 2.1]{IKra}), we have $\tau(\SU_n)\subset
s_0(\U_n/\ZZ_m)s_0^{-1}$. As the action is effective, $\tau(\SU_n)$
is isomorphic to $\SU_n$. Clearly, $s_0(\U_n/\ZZ_m)s_0^{-1}$
contains a subgroup isomorphic to $\SU_n$ if and only if $(n,m)=1$, 
in which case $\tau(\SU_n)=s_0\SU_ns_0^{-1}$, where in the right-hand side $\SU_n$ is
embedded in $ {\mathcal G}$ in the standard way. Hence, there exists an
automorphism $\gamma$ of $\SU_n$ and $D\in \GL_n(\CC)$
such that $B\{[z]\}=\{[D\gamma(B)D^{-1}z]\}$ for all
$z\in\CC^n\setminus\{0\}$ and $B\in \SU_n$. 

Next, every automorphism of $\SU_n$ has either the form
\begin{equation}
B\mapsto h_0Bh_0^{-1}\label{autoform1}
\end{equation}
or the form
\begin{equation}
B\mapsto h_0\overline{B}h_0^{-1}\label{autoform2}
\end{equation}
for some $h_0\in \SU_n$ (cf.~Remark \ref{dimtwo}). If $\gamma$ has
the form (\ref{autoform1}), then there exists $C\in \GL_n(\CC)$ such
that
\begin{equation}
B\{[z]\}=\{[CBC^{-1}z]\}\label{type1su}
\end{equation}
for all
$z\in\CC^n\setminus\{0\}$ and $B\in \SU_n$. If $\gamma$ has
the form (\ref{autoform2}), then there exists\linebreak $C\in \GL_n(\CC)$ such
that
\begin{equation}
B\{[z]\}=\{[C\overline{B}C^{-1}z]\}\label{type2su}
\end{equation}
for all
$z\in\CC^n\setminus\{0\}$ and $B\in \SU_n$.

Suppose first that $\SU_n$ acts as in (\ref{type1su}). Restrict $\tau$ to the center $Z$ of
$\U_n$. Composing $\tau$ with the projections of the group $s_0(\U_n/\ZZ_m)s_0^{-1}\cdot {\mathbf K}$ to its first and second factors, we see that there exist homomorphisms $\tau_1:Z\ra s_0(Z/\ZZ_m)s_0^{-1}=Z/\ZZ_m$ and $\tau_2: Z\ra {\mathbf K}$ such that $\tau(g)=\tau_1(g)\cdot\tau_2(g)$ for all
$g\in Z$. Clearly, we have $\tau_1(e^{it}\cdot\hbox{id})=((e^{i\sigma t}\cdot\hbox{id})H)\ZZ_m$ for some $\sigma\in\RR$. 
Further, there exists $\mu\in\RR$ such that
$\tau_2(e^{it}\cdot\hbox{id})=(d^{\mu t}H)\ZZ_m$. Since
$\tau_2(e^{it}\cdot\hbox{id})=\tau_2(e^{i(t+2\pi)}\cdot\hbox{id})$,
the number $\mu$ has to be of the form $\mu=\frac{L}{2\pi}$ for some
$L\in\ZZ$. Further, since by (\ref{type1su}) the map $\tau_2$ is trivial on the center of $\SU_n$,
we obtain that $L=nr$ for some $r\in\ZZ$. Also, as the action of $\U_n$ on $M_d^n/\ZZ_m$ is transitive, it follows that $r\ne 0$.

Next, write any $A\in \U_n$ in the form $A=e^{it}\cdot B$, where
$t\in\RR$, $B\in \SU_n$. Then for every $z\in\CC^n\setminus\{0\}$ we
have
\begin{equation}
\begin{array}{l}
A\{[z]\}=(e^{it}\cdot B)\{[z]\}=e^{it}(B\{[z]\})=\\
\vspace{-0.3cm}\\
e^{it}\{[CBC^{-1}z]\}=
\left\{\left[e^{i\sigma
  t}d^{\frac{nrt}{2\pi}}CBC^{-1}z\right]\right\}.\label{sig}
\end{array}
\end{equation}
Representing $A$ as $A=e^{i(t+\frac{2\pi k}{n}+2\pi \ell)}\cdot
(e^{-\frac{2\pi i k}{n}}\cdot B)$ with $\ell\in\ZZ$, $k\in\{0,\dots,n-1\}$, we obtain from (\ref{sig}) that
$\sigma$ must be of the form $\sigma=1+(p+\frac{q}{m})n$ for some
$p,q\in\ZZ$. This yields (\ref{type1}).

Similarly, the case when $\SU_n$ acts as in (\ref{type2su}) leads to (\ref{type2}). \end{proof}

\begin{proof}[Proof of Proposition {\rm \ref{fin}}] We start by considering action (\ref{type1}) and suppose first that it is not effective. Then for a nontrivial element $A=e^{it}\cdot B$, $t\in\RR$, $B\in \SU_n$ one has $A\{[z]\}=\{[z]\}$ for all $z\in\CC^n\setminus\{0\}$. Hence, for some $k\in\{0,\dots,n-1\}$ we have $B=e^{\frac{2\pi i k}{n}}\cdot\hbox{id}$, and there exist $\ell\in\ZZ$, $K\in\{0,\dots,m-1\}$ such that
$$
e^{i\left(1+(p+\frac{q}{m})n\right)t}d^{\frac{nrt}{2\pi}}e^{\frac{2\pi i k}{n}}=e^{\frac{2\pi i K}{m}}d^{\ell}.
$$
This yields
\begin{equation}
\begin{array}{l}
\displaystyle t=\frac{2\pi \ell}{nr},\\
\vspace{-0.3cm}\\
\displaystyle k=nL-\frac{\ell}{r}\left(1+\left(p+\frac{q}{m}\right)n\right)+\frac{nK}{m}
\end{array}\label{third}
\end{equation}
for some $L\in\ZZ$. Since $A$ is nontrivial, we have $\frac{\ell}{r}(p+\frac{q}{m})-\frac{K}{m}\not\in\ZZ$, which is
condition (b). Condition (a) follows from the second identity in (\ref{third}).

Conversely, suppose that for some $\ell\in\ZZ$, $K\in\{0,\dots,m-1\}$ conditions (a) and (b) hold. Define $t$ as in the first equation in (\ref{third}). Next, it follows from condition (a) that the right-hand side of the second equation in (\ref{third}) is an integer, thus we define $k$ by this equation with $L\in\ZZ$ chosen arbitrarily. From condition (b) we then immediately see that $A:=e^{i(t+\frac{2\pi  k}{n})}\cdot\hbox{id}$ is nontrivial. On the other hand, it is also easy to check that $A$ acts on $M_d^n$ trivially. Thus, the action of $\U_n$ on $M_d^n/\ZZ_m$ is not effective.   

The proof in the case of action (\ref{type2}) is analogous to the proof above.\end{proof}

\end{document}